\documentclass[12pt,a4paper]{article}
\usepackage{amsmath, amssymb, amsfonts, amsthm,latexsym}
\usepackage{graphicx}
\usepackage{verbatim}
\usepackage{multirow}
\usepackage{enumerate}
\usepackage[margin=40pt]{caption}
\DeclareMathAlphabet{\mathscr}{OT1}{pzc}{m}{it}

\theoremstyle{definition}

\newtheorem{theorem}{Theorem}[section]
\newtheorem{lemma}[theorem]{Lemma}
\newtheorem*{lemma*}{Lemma}
\newtheorem{definition}[theorem]{Definition}
\newtheorem*{definition*}{Definition}

\numberwithin{equation}{section}

\newtheorem*{observation*}{Observation}
\newtheorem*{proposition*}{Proposition}

\newtheorem{claim}{Claim}
\newtheorem{openproblem}{Open Problem}

\newcommand{\insertnote}[3]{#2}

\newcommand{\TODO}{\insertnote{TODO}}
\newcommand{\DOTHIS}{\insertnote{DOTHIS}}

\newcommand{\cross}{\times}
\newcommand{\tensor}{\otimes}
\newcommand{\R}{\mathbb{R}}
\newcommand{\minsym}{\wedge}

\renewcommand{\P}{\mathbb{P}}
\newcommand{\sampled}{X}
\newcommand{\resamplede}{X^{\epsilon}}

\newcommand{\factor}[2]{\mathcal{F}_{#1 #2}}
\newcommand{\commafactor}[2]{\mathcal{F}_{#1,#2}}
\newcommand{\restrict}[3]{\mathcal{R}_{{#1}{#2}}(#3)}
\newcommand{\Res}[1]{\mathcal{R}(#1)}
\newcommand{\stripleavetime}[1]{T_{#1}}
\newcommand{\reservoir}{\mathcal{G}}

\newcommand{\upperhp}{\mathcal{F}_{+}}
\newcommand{\lowerhp}{\mathcal{F}_{-}}
\newcommand{\wholefield}{\mathcal{F}}

\newcommand{\twostrips}{\lowerhp \tensor \upperhp}
\newcommand{\twostripsreservoir}{\twostrips \tensor \reservoir}

\newcommand{\webnoargs}{\phi}
\newcommand{\web}[3]{\webnoargs_{{#1}{#2}}(#3)}

\begin{document}

{
  \newcommand{\sigfield}{$\sigma$-field}
{
\title{The Brownian web is a two-dimensional black noise}

\newcommand{\tomthanks}{School of Mathematics, Raymond and Beverly Sackler Faculty of Exact
Sciences, Tel Aviv University, Tel Aviv, Israel.  This work was supported in part by
a Wyndham Deedes Memorial Travel Scholarship from The Anglo-Israel
Association.}

\newcommand{\ohadthanks}{School of Mathematics, Raymond and Beverly Sackler Faculty of Exact
Sciences, Tel Aviv University, Tel Aviv, Israel. E-mail:
ohad\_f@netvision.net.il. Research supported by an ERC advanced grant.}

\author{Tom Ellis\thanks{\tomthanks}\\%
\and Ohad N. Feldheim\thanks{\ohadthanks}}

\date{March 2012}

\maketitle

\begin{abstract}
The Brownian web is a random variable consisting of a Brownian motion
starting from each space-time point on the plane.  These are independent
until they hit each other, at which point they coalesce. Tsirelson mentions
this model in \cite{tsirelson-scaling-limit-noise-stability}, along with
planar percolation, in suggesting the existence of a two-dimensional black
noise. A two-dimensional noise is, roughly speaking, a random object on the
plane whose distribution is translation invariant and whose behavior on
disjoint subsets is independent.  Black means sensitive to the resampling of
sets of arbitrarily small total area.

Tsirelson implicitly asks: ``Is the
Brownian web a two-dimensional black noise?''.  We give a positive
answer to this question, providing the second known example of such
after the scaling limit of critical planar percolation.
\end{abstract}

\section{Introduction}
In this paper we study a stochastic object called the \emph{Brownian web}. We
research this object in the context of the theory of classical and
non-classical noises, developed by Boris Tsirelson
(see \cite{tsirelson-nonclassical-stochastic-flows} for a survey).
Our main result
is that, in the terminology of this theory, the Brownian web is a
two-dimensional \emph{black} noise.
Roughly speaking, the Brownian web is a random variable which assigns to
every space-time point in $\R\times\R$ a standard Brownian motion starting
at that point.  The motions in each finite subcollection are independent
until the first time that one hits another
and from thereon those two coalesce, continuing together. This object was
originally studied more then twenty-five years ago by Arratia \cite{arratia}, motivated
by a study of the asymptotics of one-dimensional voter models, and later
by T\'{o}th and Werner \cite{toth-werner},
motivated by the problem of constructing continuum
``self-repelling motions'', by Fontes, Isopi, Newman and Ravishankar
\cite{fontes-et-al},
motivated by its relevance to ``aging'' in statistical physics of
one-dimensional coarsening, and by Norris and Turner
\cite{norris-turner}
regarding a scaling limit of a two-dimensional aggregation process.
A rigorous notion of the Brownian web in our context
can be found in \cite{tsirelson-lecture-course} for the case of coalescing
Brownian motions on a circle.  The above also provide
their own constructions of the Brownian web.

The Brownian web functions as an important example in the theory of
classical and non-classical noises. In this
theory a noise is a probability space equipped with a collection
of sub-\sigfield{}s indexed by the open rectangles (possibly infinite) of
$\R^d$.  The sub-\sigfield{} associated to a rectangle is intended to
represent the behavior of a stochastic object within that rectangle.
The \sigfield{}s must satisfy the following three properties:
\begin{itemize}
\item the \sigfield{}s associated to disjoint rectangles of $\R^d$ are
independent,
\item translations on $\R^d$ act in a way that preserves the
probability measure,
\item the \sigfield{}
associated to a rectangle is generated by the two \sigfield{}s
associated to two smaller rectangles which partition it.
\end{itemize}
Two natural examples of noises are the Gaussian white noise
and the Poisson noise. These noises are called classical, or white,
meaning that
resampling of a small portion of $\R^d$ doesn't change the observables of the
process very much.

The foundational result of Tsirelson and Vershik \cite{tsirelson-vershik} showed that there
exist non-classical noises.  Indeed they showed the existence of non-classical noises
that are as far from white as could be, and these are called black.  The
defining property of a black noise is that all its observables are
sensitive, i.e.\ for any particular observable, resampling
a small scattered portion of the noise
renders that observable nearly independent of its original
value.  (For a thorough discussion of black and white, classical
and non-classical noises see \cite{tsirelson-nonclassical-stochastic-flows}).
Tsirelson showed in \cite{tsirelson-scaling-limit-noise-stability}
(Theorem 7c2) that the Brownian web, when considered as a time-indexed
random process, is one-dimensional black noise.  We extend this result
by showing:

\begin{theorem}
\label{thm:bw-2d-black-noise}
The Brownian web is a
two-dimensional black noise.
\end{theorem}

This makes the Brownian web only the second
known two-dimensional black noise after Schramm and Smirnov's
recent result on the scaling limit of critical planar
percolation \cite{schramm-smirnov}.

Of the three properties required for a process to be a noise, the
first two, i.e.\ translation invariance and independence on disjoint
domains, hold trivially for the Brownian web.  Furthermore, once we
have shown that the Brownian web is a two-dimensional noise it will
follow that it is a two-dimensional \emph{black} noise,
through a general argument.

The main difficulty in proving
Theorem \ref{thm:bw-2d-black-noise} is to show that the $\sigma$-field
associated to any rectangle is generated by the two $\sigma$-fields
associated to any two rectangles that partition it.
The major milestone towards this result is to show the special case
when the large rectangle is the whole plane, and the smaller
rectangles that partition it are the upper and lower half-planes.

\begin{theorem}
\label{thm:informal-recovering-from-half-planes}
In the Brownian web, the $\sigma$-field associated to the whole plane
is generated by that associated to the upper half-plane and that
associated to the lower half-plane.
\end{theorem}

The rest of the paper goes as follows: in Section \ref{sec:brownian-web-definition} we define the
Brownian web formally; in Section \ref{sec:recovering-from-half-planes} we restate
Theorem \ref{thm:informal-recovering-from-half-planes} as
Theorem \ref{thm:recoveringfromhalfplanes}; we then reduce this theorem to a
convergence result for an auxiliary process which we prove in Section
\ref{sec:proof-of-lem:resamplede-to-sampled}.  In Section \ref{sec:recovering-from-strips} we extend Theorem \ref{thm:informal-recovering-from-half-planes} to hold for the
$\sigma$-fields associated to horizontal strips as well; in Section \ref{sec:conclusions-about-the-noise}
we extend further to all rectangles.  In addition we define noises
properly and conclude by proving Theorem \ref{thm:bw-2d-black-noise}.  Section \ref{sec:open-problems} is devoted to
remarks, open problems and acknowledgements.
}

  {
\section{Definition of the Brownian web}
\label{sec:brownian-web-definition}

The Brownian web is the continuum scaling limit of a system of
independent-coalescing random walks (see
\cite{tsirelson-lecture-course}).  Constructing the
continuum version raises several technical difficulties addressed in
\cite{fontes-et-al},\cite{norris-turner},\cite{toth-werner}.
Nonetheless, all the
constructions share the following property, which we use as a
definition:

\newcommand{\simplex}{\mathcal{S}}

  Denote $\simplex=\{(s, t) \in \R^2 : s \le t\}$.
  A Brownian web on a probability space $\Omega$ is a (jointly)
  measurable mapping $\webnoargs : \Omega \cross \simplex \cross \R
  \to \R$, $(\omega, (s, t), x) \mapsto \web{s}{t}{x}$ (suppressing
  $\omega$ in the notation) such that for every finite collection of
  starting points $(s_1, x_1),(s_2, x_2),...,(s_n, x_n)$, the
  collection of processes $\web{s_1} {\cdot}{x_1},
  \web{s_2}{\cdot}{x_2},...,\web{s_n}{\cdot}{x_n}$
  forms a system of $n$ independent-coalescing Brownian motions.

  \newcommand{\bm}[1]{X^{#1}}

  A system of $n$ independent-coalescing Brownian motions is a finite
  collection of stochastic processes $(\bm{1}, \bm{2},...,\bm{n})$ such that
  each $\bm{i}$ starts at some point $x_i$ at some time $s_i$, and $(\bm{1},
  \bm{2},...,\bm{n})$ are independent until the first time $T$ at which
  $\bm{i}(T)=\bm{j}(T)$ for some $i\neq j$. From this time onwards $\bm{i}(T)$
  and $\bm{j}(T)$ coalesce and continue with the rest of the $\bm{k}$ (for
  $k\neq i,j$) as a system of $n-1$ independent-coalescing Brownian motions.
  Several trajectories of a Brownian web can be seen in Figure
  \ref{fig:bw-trajectories}.

\begin{figure}
   \centering
   \includegraphics[scale=2]{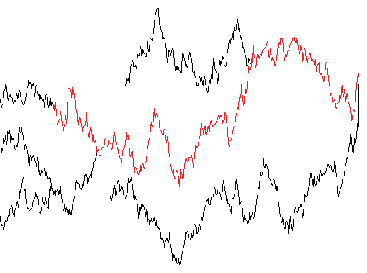}
   \caption{Some trajectories of the Brownian web. A particular trajectory is marked.}
  \label{fig:bw-trajectories}
\end{figure}
}

  \newcommand{\indepbm}{\psi}
\newcommand{\toinP}{\overset{\P}\to}
\newcommand{\statementoflemresampledetosampled}{$\resamplede \toinP \sampled$ as $\epsilon \to 0$}

\newcommand{\statewebO}{S_{\webnoargs}}
\newcommand{\statenowebO}{S_{\indepbm}}
\newcommand{\trajs}{\mathcal{X}}
{
\section{Recovering the web from half-planes}
\label{sec:recovering-from-half-planes}

We introduce three \sigfield{}s generated by the Brownian web, and use
them to restate Theorem \ref{thm:informal-recovering-from-half-planes}
as Theorem \ref{thm:recoveringfromhalfplanes}.

\newcommand{\restrictupper}{\mathcal{R}_{+}}

  Write $\trajs$ for the collection of trajectories comprising the whole web,
  that is $\{t \mapsto \web{s}{t}{x} : (s,x) \in \R^2\}$,
  and $\wholefield$ for the $\sigma$-field generated by the
  whole web, i.e.\ generated by $\trajs$.
  We further introduce $\upperhp$ and $\lowerhp$, the \sigfield{}s
  generated by the web on the upper and lower half-planes
  respectively.  Formally,

  \begin{definition*}
  For any path $f$ we write $\restrictupper(f)$ for $f$ stopped at the
  first time it is outside the upper half-plane.
  $\restrictupper(t \mapsto \web{s}{t}{x})$ is therefore the trajectory of the
  web $\webnoargs$ started at the point $(s,x)$ and stopped at the first
  time it is outside the upper half-plane (if $(s,x)$ is outside the
  upper half-plane, $f$ is stopped immediately).
  We define $\upperhp$ to be the $\sigma$-field generated by the
  collection of paths $\{\restrictupper(X) : X \in \trajs \}$, and define $\lowerhp$ analogously.
  \end{definition*}

  Note that $\upperhp$ and $\lowerhp$ are independent by the definition of
  a system of $n$ independent-coalescing Brownian motions.

  \newcommand{\F}{\mathcal{F}}
  For $\sigma$-fields $\F_a$, $\F_b$, $\F_c$ we write $\F_a = \F_b
  \tensor \F_c$ when $\F_a$
  is generated by $\F_b$ and $\F_c$ (up to sets of measure $0$),
  and $\F_b$ and $\F_c$ are independent.
  We are now ready to restate Theorem
  \ref{thm:informal-recovering-from-half-planes}.

\begin{theorem}\label{thm:recoveringfromhalfplanes}
  $\wholefield = \twostrips$.
\end{theorem}

The difficulty is to show that $\wholefield \subseteq \twostrips$,
because the reverse inclusion follows directly from the definitions.
To avoid having to state results as ``for all $\sampled \in \trajs$
\ldots'', here and for the rest of the paper we write $\sampled$ for
an arbitrary element of $\trajs$.  This we do purely for the sake of
notational simplicity.

$\wholefield$ is generated by such processes, so
our theorem reduces to the following lemma.

\begin{lemma}
  \label{lem:sampled-twostrip-meas}
  $\sampled$ is $\twostrips$-measurable.
\end{lemma}

To prove this we must show that $\sampled$ can be recovered using
trajectories starting in the upper half-plane which stop when they hit
$0$, and trajectories starting on the lower half-plane which stop when
they hit $0$.

To do so, we might have liked to recover $X$ (starting on, say, the
upper half-plane) by following it until it hits 0, then by following its
continuation within the lower half-plane.  However, this is impossible
since when trajectories of Brownian motion cross $0$ they do so
infinitely often within a finite period of time.

To overcome this problem we introduce in Section
\ref{subsec:the-perturbed-process} a process $\resamplede$, for each
$\epsilon > 0$, which is $\twostrips$-measurable.  In Section
\ref{sec:proof-of-lem:resamplede-to-sampled} we show that
$\resamplede$ is an approximation of $\sampled$ in the sense that
$\resamplede$ converges to $\sampled$ (in probability) as $\epsilon \to 0$.  This will
imply that $\sampled$ itself is $\twostrips$-measurable.

\subsection{The perturbed process}
\label{subsec:the-perturbed-process}

\TODO{}{Should we call this the approximation process?}

{
\newcommand{\joinernoargs}{\psi}
\newcommand{\joiner}[2]{\joinernoargs_{{#1}{#2}}}
\newcommand{\joinerval}[1]{\joinernoargs_{#1}}
  From our arbitrary $\sampled$ we now construct $\resamplede$, a
  ``perturbed'' version of $\sampled$, which depends also on
  $\joinerval{t}$,
  some Brownian motion independent of $\webnoargs$ (measurable with
  respect to $\reservoir$, say, where $\reservoir$ is independent of
  $\wholefield$).

  In the definition of $\resamplede$
  we give the word ``follows'' two distinct meanings.
  We say $\resamplede$ follows $\webnoargs$ on $[s,u]$ if
  $\resamplede_t = \web{s}{t}{\resamplede_s}$ for $t \in [s,u]$.
  That is if the trajectory of $\resamplede$ follows the trajectory of
  the web starting from point $\resamplede_s$ at time $s$ and up to time $u$.
  We say $\resamplede$ follows $\joinernoargs$ on $[s,u]$ if
  $\resamplede_t = \resamplede_s + \joinerval{t} - \joinerval{s}$ for $t \in [s,u]$.

\begin{definition}
  \label{def:resamplede}
  The perturbed process $\resamplede$ starts at the same time and position
  as $\sampled$ and
  alternates between two states.
  In state $\statewebO$ it follows
  $\webnoargs$ while, in state $\statenowebO$ it follows $\indepbm$.
  The process starts in state $\statewebO$ and
  the transition from state $\statewebO$ to state $\statenowebO$ occurs
  when $\resamplede$ hits $0$, while
  the transition from state $\statenowebO$ to state $\statewebO$ occurs when
  $\resamplede$ hits $\pm \epsilon$.
  \TODO{}{State why this definition is well-defined?}
  See Figure \ref{fig:perturbed} for an illustration of sample paths of
  $\sampled$ and $\resamplede$.
\end{definition}
}

\begin{figure}
   \centering%
   $\begin{array}{cc}
   \includegraphics[scale=0.33]{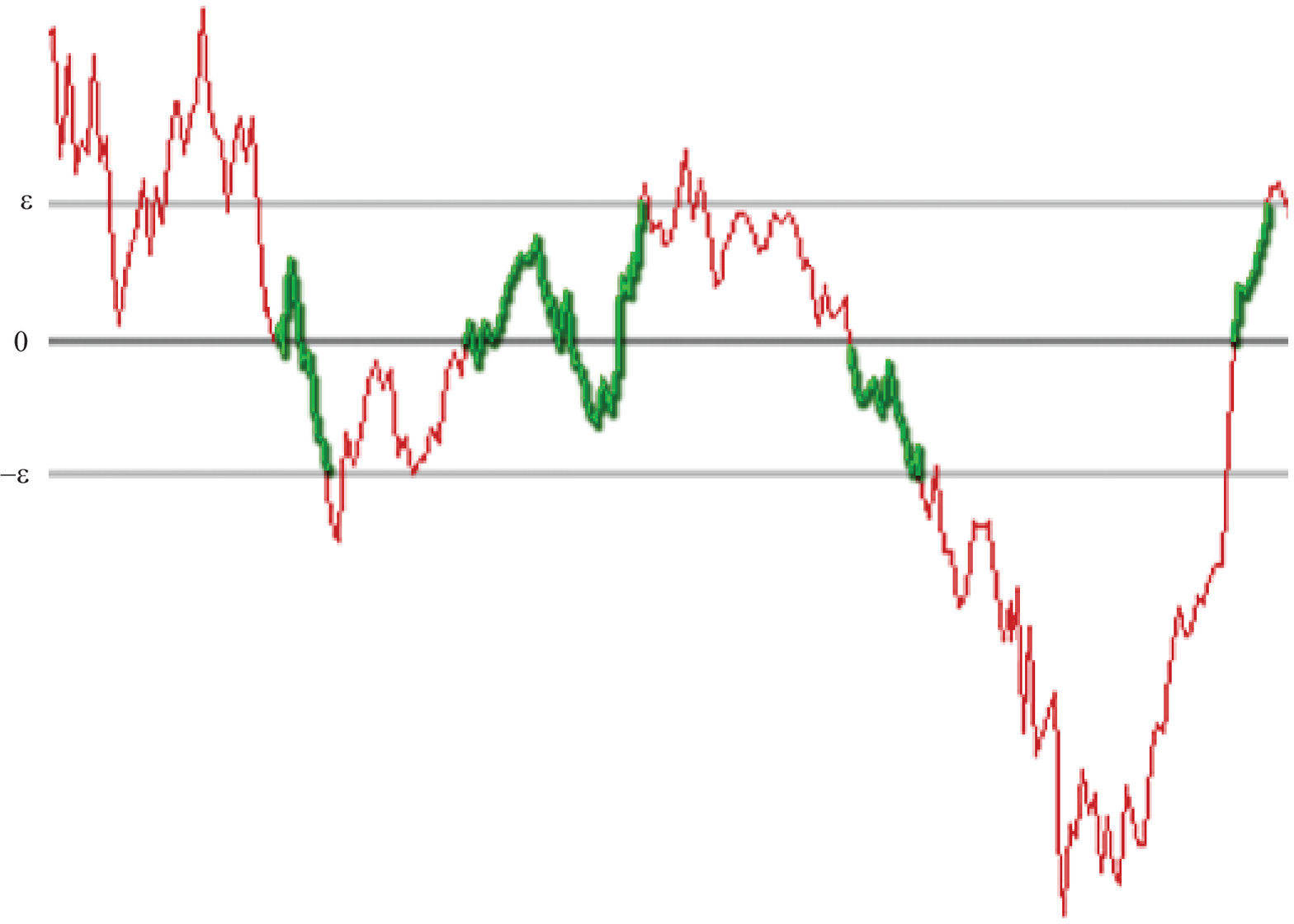}&
   \includegraphics[scale=0.33]{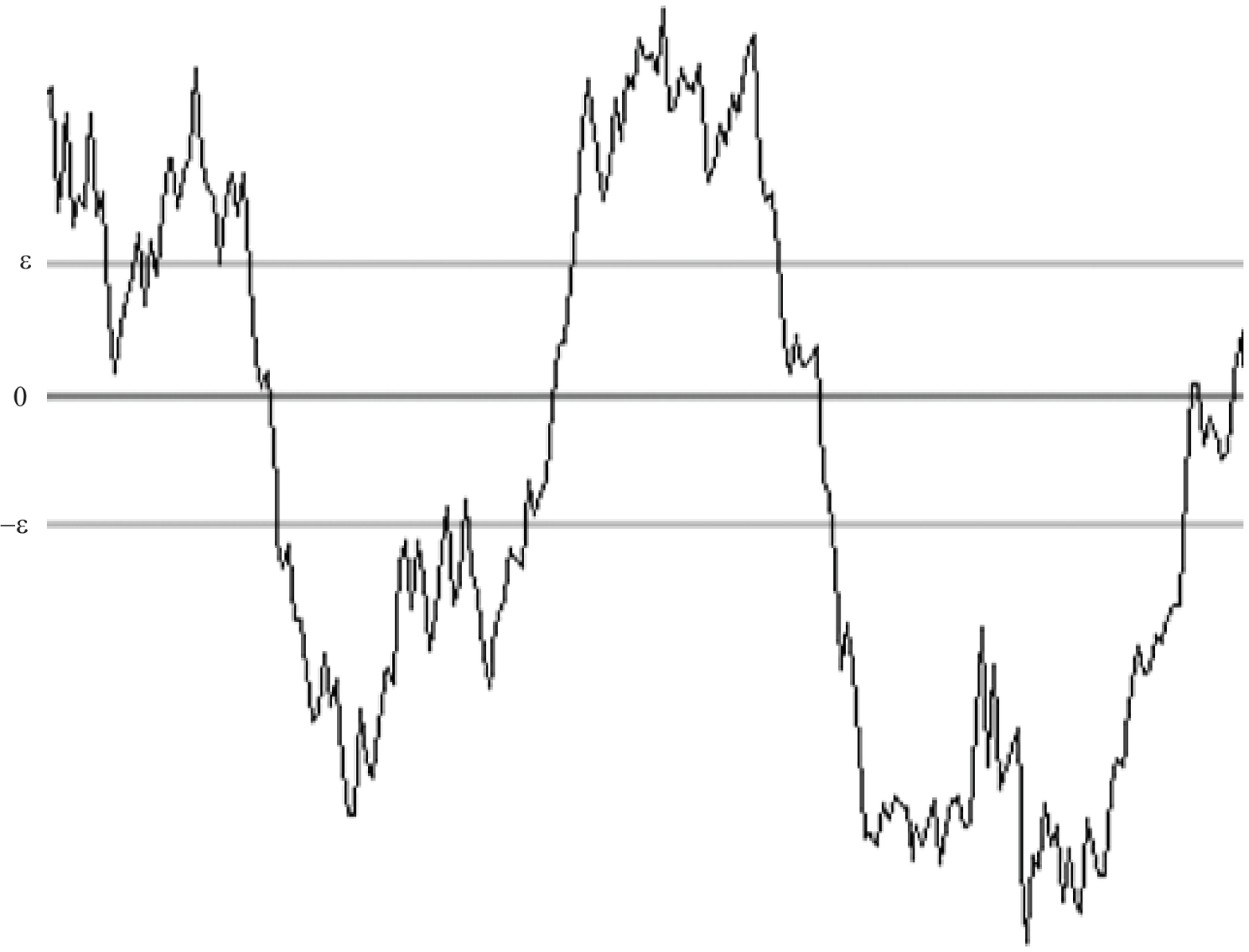}\\
   \includegraphics[scale=0.66]{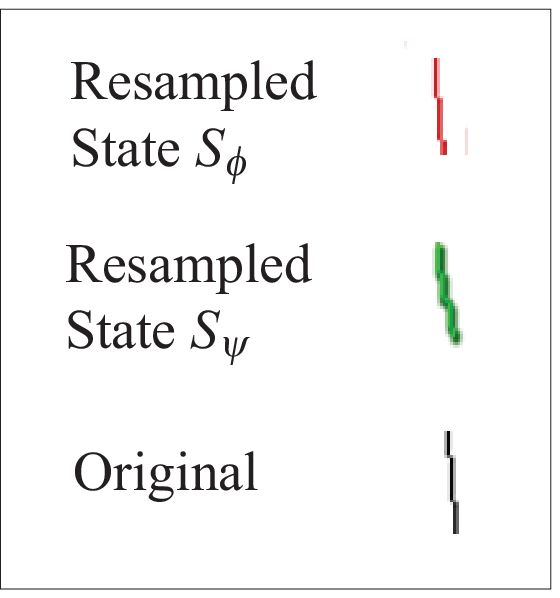}&
   \includegraphics[scale=0.33]{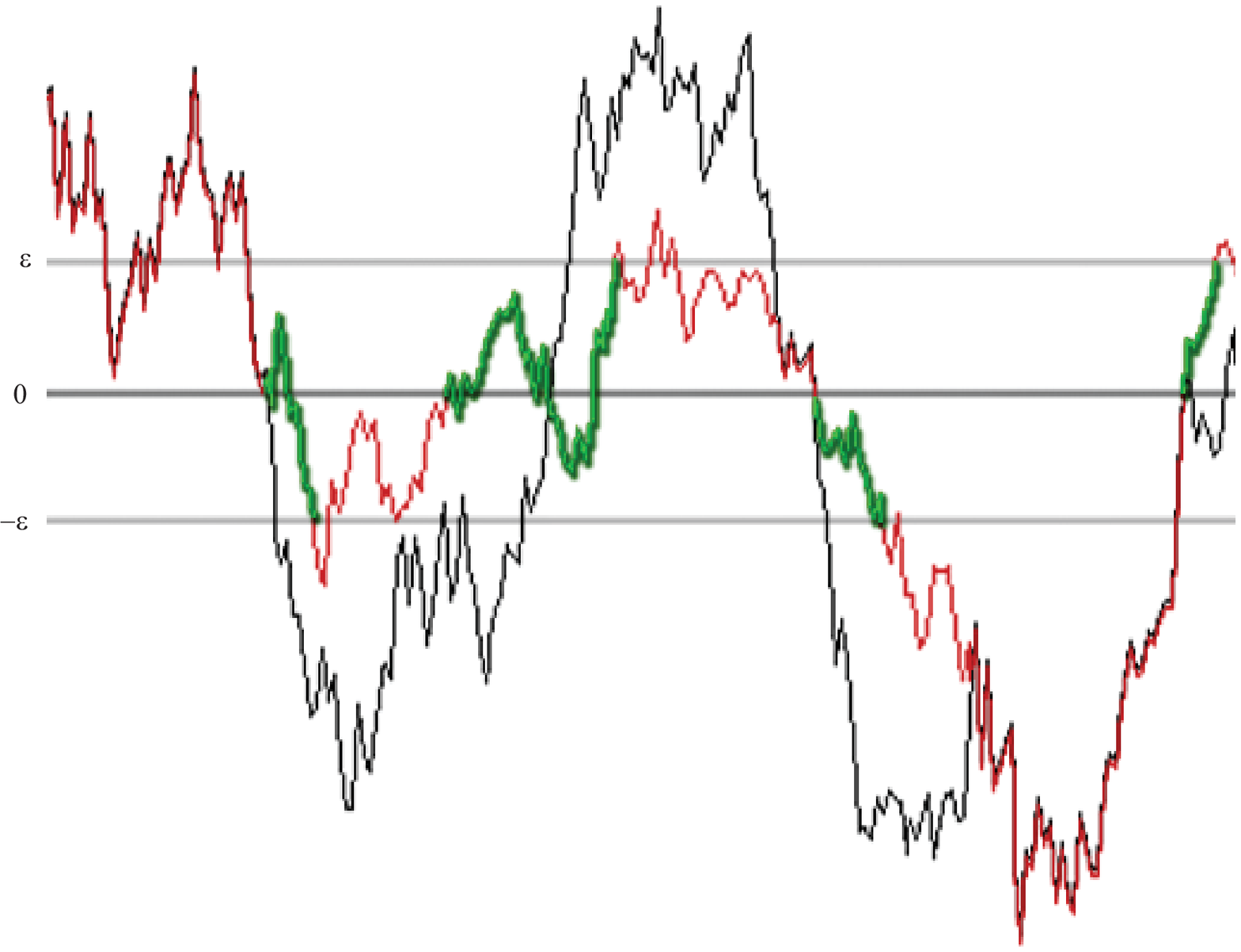}
   \end{array}$
   \caption{
     \label{fig:perturbed}
   A sample of a web trajectory and the corresponding perturbed processes.
   The top left image depicts the perturbed process, with state $\statenowebO$ in bold.
   The top right image depicts the original web trajectory. The center image
   illustrates both processes together, showing the segments where they coalesce.
   }
\end{figure}

The following lemma specifies in what sense the perturbed process is
an approximation of the trajectory of the web.
Here and in the rest of the paper the convergence is uniform
on compacts in probability.

\begin{lemma}
  \label{lem:resamplede-to-sampled}
  \statementoflemresampledetosampled.
\end{lemma}
Lemma \ref{lem:sampled-twostrip-meas} reduces to Lemma
\ref{lem:resamplede-to-sampled}.

\begin{proof}[Proof of reduction]
  $\resamplede$ is $\twostripsreservoir$-measurable, so we use Lemma
  \ref{lem:resamplede-to-sampled} to conclude that $\sampled$ is also
  $\twostripsreservoir$-measurable.  However, since $\sampled$ is
  actually independent of $\reservoir$ we can use a basic result on
  tensor products of Hilbert spaces (for example Equation (2c8) of
  \cite{tsirelson-noise-as-a-boolean-algebra}) to conclude that $\sampled$ is
  in fact $\twostrips$-measurable.
\end{proof}

We devote the following section to the proof of Lemma
\ref{lem:resamplede-to-sampled}.
}

  {
\section{Convergence of the perturbed process}
\label{sec:proof-of-lem:resamplede-to-sampled}

\newcommand{\statenoweb}{S^{2D}_{\indepbm}}
\newcommand{\statewebapart}{S^{2D}_{\webnoargs}}
\newcommand{\statewebtogether}{S^{1D}_{\webnoargs}}
\newcommand{\twodim}{Y}

\newcommand{\twodime}{\twodim(\epsilon)}

In this section we prove that
\statementoflemresampledetosampled{}.  This statement depends only on the
joint distribution of $\sampled$ and $\resamplede$.
We therefore define $\twodime = \twodim=(\resamplede, \sampled)$ (generally suppressing
the $\epsilon$ dependence in the notation). Let us describe the distribution of
$\twodim$
 as a two-dimensional random process.

We classify the behavior of the process into three states according to
the behavior of $\resamplede$ with respect to $\sampled$.
\begin{itemize}
\item If $\resamplede$ is in state $\statewebO$
and is coalesced with $\sampled$ we say $\twodim$ is in
state $\statewebtogether$.
\item If $\resamplede$ is in state $\statewebO$
and \emph{is not} coalesced with $\sampled$ we say $\twodim$ is in state $\statewebapart$.
\item If $\resamplede$ is in state $\statenowebO$
we say $\twodim$ is in state $\statenoweb$.
\end{itemize}
$\twodim$ starts
in $\statewebtogether$.  From $\statewebtogether$, $\twodim$
can only transition to $\statenoweb$.
This transition occurs when $\twodim$ hits the origin, as the coalesced
$\sampled$ and $\resamplede$ will continue together until they leave their
current half-plane.
From $\statenoweb$, $\twodim$ can only transition to
$\statewebapart$.  This transition occurs when $\resamplede$
leaves the $(-\epsilon,\epsilon)$
interval (i.e.\ $\twodim$ hits either of the $x=\pm\epsilon$ lines).
From
$\statewebapart$, $\twodim$ can either transition to $\statewebtogether$
if $\sampled$ and $\resamplede$ coalesce (i.e.\ $\twodim$ hits the line $x=y$)
or transition to $\statenoweb$ if $\resamplede$ hits $0$ (i.e.\ $\twodim$ hits
the $x=0$ line).
States and possible transitions of $\twodim$ are summarized in Figure
\ref{fig:twodimtranstab}.

The form of the labels given to the states is justified by the following.

\begin{observation*}
In $\statewebtogether$, $\twodim$ follows the law of a (time scaled) one-dimensional
Brownian motion on the line $x = y$.
In $\statewebapart$ and $\statenoweb$, $\twodim$ follows the law of a
two-dimensional Brownian motion.
\end{observation*}

Additionally observe that by the scale-invariance of Brownian motion,
the distribution of the sample paths of $\twodime/\epsilon$ is
independent of $\epsilon$.

\begin{figure}
\begin{center}
%  \begin{tabular}{c | l || c | c | c | c | }
\renewcommand{\arraystretch}{0.9}
\begin{tabular}{c|c|c|c|c|c|}
\cline{2-6}
 & State & Illustration & Law & Next & Trans. Cond. \\ \cline{1-6}
\multicolumn{1}{|c|}{\multirow{18}{*}{$\resamplede$}} &
\multicolumn{1}{|c|} {\multirow{6}{*}{$\statewebtogether$}} &  & \multirow{6}{*}{equal} & \multirow{6}{*}{$\statenoweb$} & \multirow{6}{*}{hits $0$}     \\
\multicolumn{1}{|c|} {} & {} & {\includegraphics[scale=0.33]{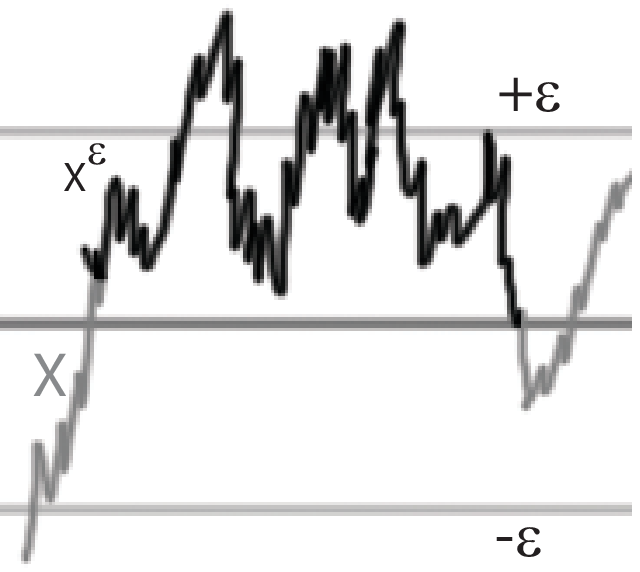}} & {} & {} &     \\ \cline{2-6}
\multicolumn{1}{|c|} {} &  \multirow{6}{*}{$\statewebapart$} &  & \multirow{6}{*}{indep.} & \multirow{3}{*}{$\statenoweb$} & \multirow{3}{*}{hits   $0$}\\
\multicolumn{1}{|c|} {} & {} & {\includegraphics[scale=0.33]{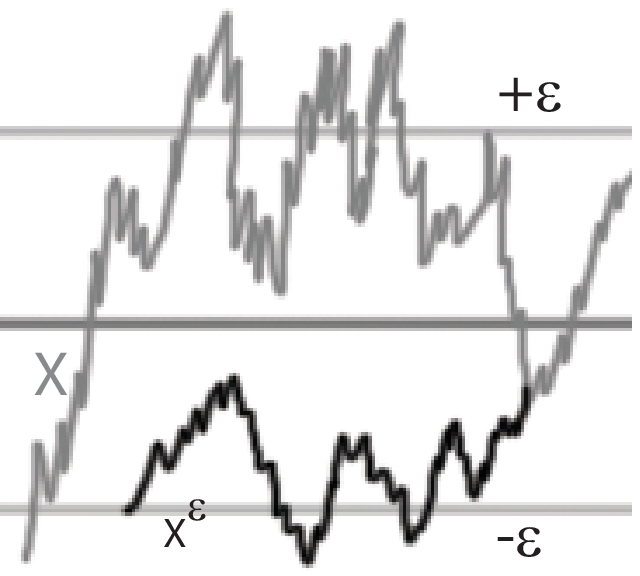}} & {} & \multirow{-3}{*}{$\statewebtogether$} &   \multirow{-3}{*}{hits  $\sampled$}  \\ \cline{2-6}
\multicolumn{1}{|c|} {} & {\multirow{6}{*}{$\statenoweb$}} & {}& \multirow{6}{*}{indep.} & \multirow{6}{*}{$\statewebapart$} & \multirow{6}{*}{hits $\pm\epsilon$}     \\
\multicolumn{1}{|c|} {} & {} & {\includegraphics[scale=0.33]{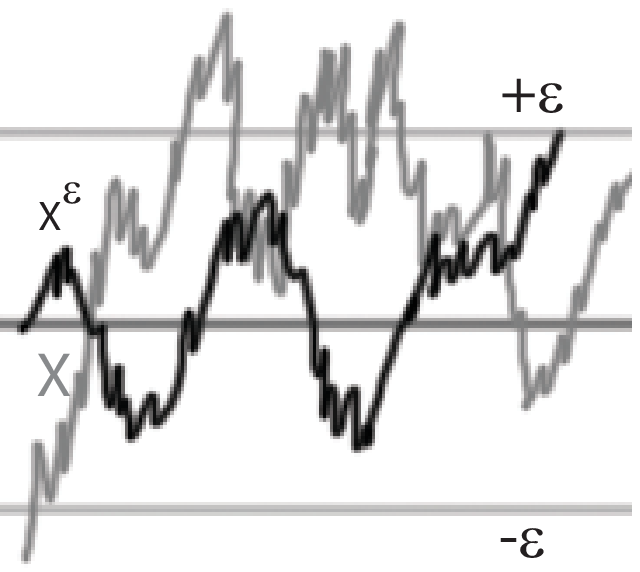}} & {} & {} &      \\ \hline\hline %\cline{1-6}
\multicolumn{1}{|c|}{\multirow{20}{*}{$\twodim=(x,y)$}} &
\multicolumn{1}{|c|} {\multirow{6}{*}{$\statewebtogether$}} &  & \multirow{6}{*}{equal} & \multirow{6}{*}{$\statenoweb$} & \multirow{6}{*}{$x=y=0$}     \\
\multicolumn{1}{|c|} {\multirow{10}{*}{$x=\resamplede$}} & {} & {\includegraphics[scale=0.33]{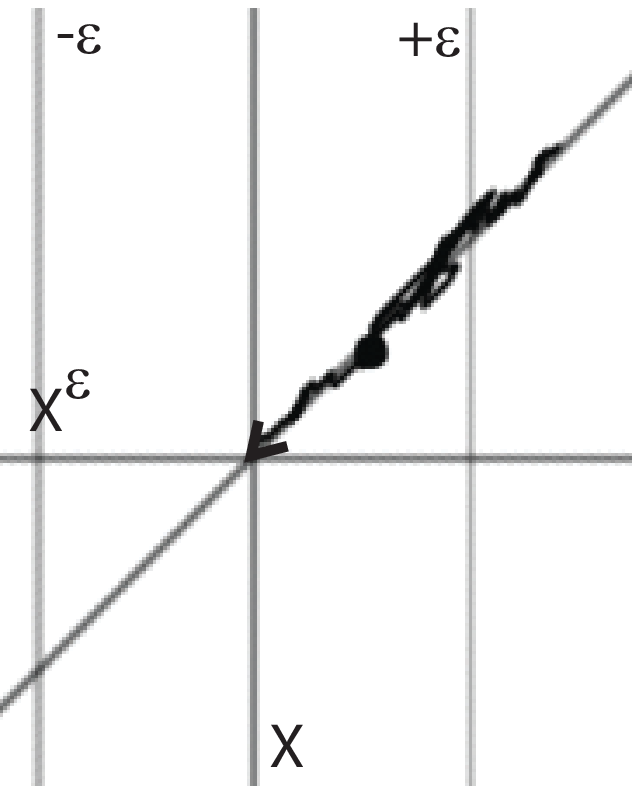}} & {} & {} &     \\ \cline{2-6}
\multicolumn{1}{|c|} {\multirow{10}{*}{$y=\sampled$}} &  \multirow{6}{*}{$\statewebapart$} &  & \multirow{6}{*}{indep.} & \multirow{3}{*}{$\statenoweb$} & \multirow{3}{*}{$x=0$}\\
\multicolumn{1}{|c|} {} & {} & {\includegraphics[scale=0.33]{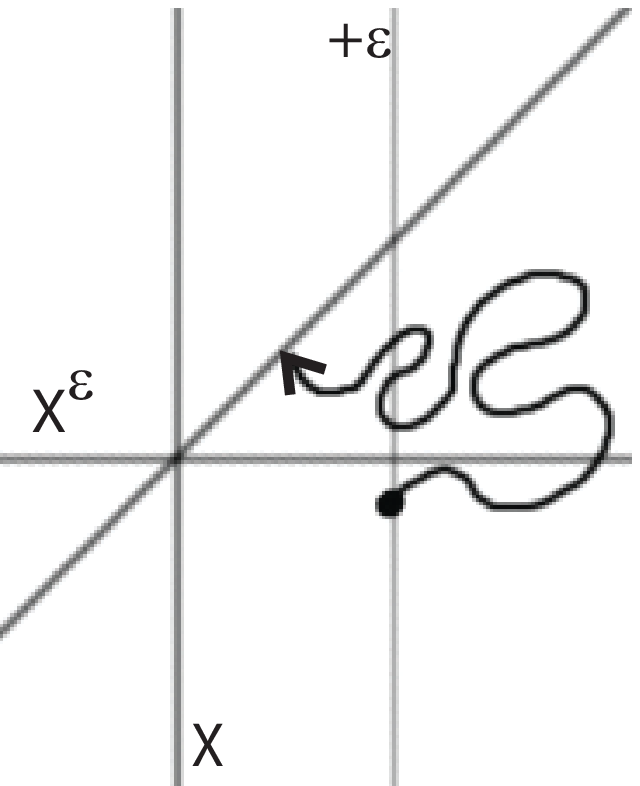}} & {} & \multirow{-3}{*}{$\statewebtogether$} &   \multirow{-3}{*}{$x=y$}  \\ \cline{2-6}
\multicolumn{1}{|c|} {} & {\multirow{6}{*}{$\statenoweb$}} & {}& \multirow{6}{*}{indep.} & \multirow{6}{*}{$\statewebapart$} & \multirow{6}{*}{$x=\pm\epsilon$}     \\
\multicolumn{1}{|c|} {} & {} & {\includegraphics[scale=0.33]{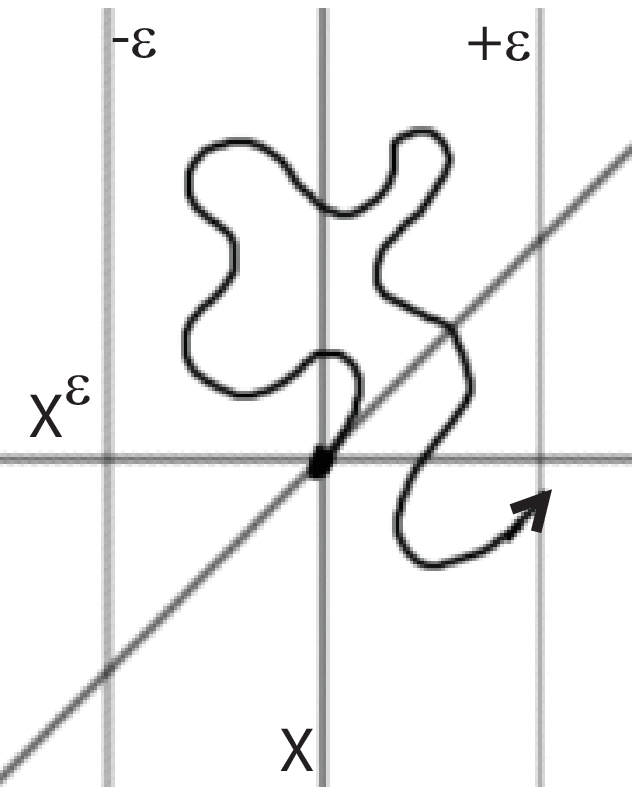}} & {} & {} &    \\
     \hline
  \end{tabular}
\end{center}
\caption{Illustrated states and transitions of $\twodim$}
\label{fig:twodimtranstab}
\end{figure}

\newcommand{\boundarylines}{A_\delta}

Define $\boundarylines=\{(x,y) \ :\  |x-y|=\delta \}$.
To prove Lemma \ref{lem:resamplede-to-sampled} we use the following
property of $\twodim$:

\newcommand{\farpoint}{(P,P)}
\newcommand{\probhitboundaryis}[1]{For given $P > 0$, $\delta > 0$ the probability that $\twodim$
  hits $\boundarylines$ before it hits $\farpoint$ is #1}

\begin{lemma}\label{lem:prob-hit-boundary-o1}
  \probhitboundaryis{$o(1)$ as $\epsilon \to 0$}.
\end{lemma}

We delay the proof of Lemma \ref{lem:prob-hit-boundary-o1} to Section
\ref{subsec:excursions-of-twodim}.

\begin{proof}[Proof of Lemma \ref{lem:resamplede-to-sampled}]

Write $s$ for the time at which $\twodim$ starts.
Lemma \ref{lem:resamplede-to-sampled} is equivalent to: for all $\delta > 0$, $u > s$
\[
\P(\twodim_t \in \{(x,y) \ :\  |x-y|<\delta \} \text{ for all } t \in [s,u])\to1 \text{ as } \epsilon\to 0.
\]
The above statement can be rephrased as

\vspace{12pt}
For all $\delta > 0$, $u > s$,
the probability that before time $u$, $\twodim$ has hit $\boundarylines$
is $o(1)$ as $\epsilon \to 0$.
\DOTHIS{}{Make this display better}
\vspace{12pt}

  We prove this as follows.  For any $\eta > 0$,
  choose $P$ so that the probability that standard Brownian motion
  travels from $0$ to $P$ in time less than $u-s$ is less than
  $\eta$.
  Apply Lemma \ref{lem:prob-hit-boundary-o1} to
  choose $\epsilon_0$ such that, for all $\epsilon < \epsilon_0$, the probability that $\twodime$ hits
  $\boundarylines$ before $\farpoint$ is less than $\eta$.
  Then the probability that $\twodime$ hits $\boundarylines$ before $\farpoint$
  or takes less time than $u-s$ to reach $\farpoint$ is
  less than $2\eta$.
  Thus the probability that $\twodime$ hits $\boundarylines$ before
  time $u$ is less than $2\eta$.
\end{proof}

\subsection{Excursions of $\twodim$}
\label{subsec:excursions-of-twodim}

In this section we prove the following, which is slightly stronger than
Lemma \ref{lem:prob-hit-boundary-o1}.

\newcommand{\loger}{\log 1/\epsilon}

\begin{lemma}\label{lem:prob-hit-boundary-o1loge}
  \probhitboundaryis{$O(\frac{1}{\loger})$}.
\end{lemma}

\newcommand{\origin}{(0,0)}

\newcommand{\excursionstart}{T}
  We begin by introducing the notion of an excursion of $\twodim$.
  Almost surely, the times at which $\twodim = \origin$ (which are stopping
  times) form an infinite discrete collection $\excursionstart_0 <
  \excursionstart_1 < \cdots$. We say ``the probability that an excursion
  hits a set $U$ is $p$'' if $\P(\twodim_t\in U \text{
  for some } t\in [\excursionstart_0, \excursionstart_{1}]) = p$.
  \TODO{}{Provide a justification for the well-definedness of these
    stopping times, perhaps in terms of the states?}
  Observe that by equidistribution this probability is the same when
  $t$ ranges over $[\excursionstart_i, \excursionstart_{i+1}]$, and
  note that the hitting events in question are independent.

\newcommand{\probexcursionhits}[1]{\P\left(\text{an excursion hits } #1\right)}

Our approach to proving Lemma \ref{lem:prob-hit-boundary-o1loge} is to
demonstrate that
\[
\probexcursionhits{\farpoint} \gg
\probexcursionhits{\boundarylines} \text { as } \epsilon \to 0.
\]
This is
realized through the next pair of lemmas whose proofs we delay until Section
\ref{proofs-of-the-excursion-lemmas}.

\newcommand{\Omegaeloge}{\Omega(\epsilon\loger)}

\begin{lemma}
  \label{lem:Phitboundaryline}
  For given $\delta > 0$, $\probexcursionhits{\boundarylines}$ is $O(\epsilon)$.
\end{lemma}

\begin{lemma}
  \label{lem:Pabsorbedandtravelsfar}
  For given $P > 0$, $\probexcursionhits{\farpoint}$ is $\Omegaeloge$.
\end{lemma}

\newcommand{\Oe}{O(\epsilon)}

\begin{proof}[Proof of Lemma \ref{lem:prob-hit-boundary-o1loge}]
  $\twodim$ consists of a sequence of excursions, each of which satisfies
  exactly one of the following conditions
  \begin{itemize}
  \item the excursion hits $\boundarylines$ (with probability
    $O(\epsilon$)),
  \item the excursion does not hit $\boundarylines$ but does hit
    $\farpoint$ (with probability $\Omegaeloge-\Oe$, which is itself
    $\Omegaeloge$),
  \item the excursion does not hit $\boundarylines$ or $\farpoint$.
  \end{itemize}
  The excursions are independent, so the probability that $\twodim$
  hits $\boundarylines$ before $\farpoint$ is
  \[
  \frac{\Oe}{\Omegaeloge + \Oe} = O\left(\frac{1}{\loger}\right).
  \]
\end{proof}

\subsection{Proofs of the excursion lemmas}
\label{proofs-of-the-excursion-lemmas}
\newcommand{\tdh}{\rotproc^1}
\newcommand{\tdv}{\rotproc^2}
\newcommand{\rotproc}{Z}

\DOTHIS{}{Perhaps should introduce more consistency between $\rotproc$
  and $\twodim$ in this proof}
{
\newcommand{\x}{\resamplede}
\newcommand{\y}{\sampled}
In this section we give the proof of Lemmas \ref{lem:Phitboundaryline} and
\ref{lem:Pabsorbedandtravelsfar}. For convenience we rotate (and scale)
$\twodim=(\x,\y)$, defining
\[\rotproc(\epsilon) = \rotproc=(\tdh,\tdv)=\frac{1}{2}(\x+\y,\x-\y).\]

As for $\twodim$, $\rotproc$ has the following ``scale invariance''
property: the distribution of sample paths of $\rotproc(\epsilon) /
\epsilon$ is independent of $\epsilon$.

\begin{proof}[Proof of Lemma \ref{lem:Phitboundaryline}]
Consider the process $\twodim$ between times $\excursionstart_0$ and
$\excursionstart_1$.  Once $\twodim$ arrives at $\statewebtogether$ it
can never hit $\boundarylines$ before hitting $\origin$.  Thus our
goal is to show that with probability at least $1-O(\epsilon)$,
$\twodim$ arrives in $\statewebtogether$ before hitting
$\boundarylines$.  Next, we observe the
following two auxiliary claims:

\begin{claim}\label{cl:tdv-together-estimate}
  Whenever $\tdv=0$ the probability that subsequently $\twodim$
  arrives at $\statewebtogether$ before $\tdv$ hits $\pm\epsilon/2$ is
  a constant (independent of $\epsilon$).
\end{claim}

\begin{claim}\label{cl:tdv-hitting-back-0-estimate}
  Whenever $\tdv=\pm\epsilon/2$ then there is probability equal to
  $\epsilon/\delta$ of $\tdv$ hitting $\pm\delta/2$ before it hits $0$.
\end{claim}

Claim \ref{cl:tdv-together-estimate} follows from scale invariance, while
Claim \ref{cl:tdv-hitting-back-0-estimate} is a standard martingale result on
Brownian motion (observing that on the relevant time interval $\tdv$ is a
standard Brownian motion).

The reduction of Lemma \ref{lem:Phitboundaryline} to those two claims is
similar to the proof of Lemma \ref{lem:prob-hit-boundary-o1loge}. Claims
\ref{cl:tdv-together-estimate} and \ref{cl:tdv-hitting-back-0-estimate} imply
that between a time when $\tdv = 0$ and the next time that
$\tdv = 0$ after having hit $\pm\epsilon/2$,
\[
\frac{\P(\twodim\text{ hits }\boundarylines)}{\P(\twodim\text{arrives at }\statewebtogether)}
\le \frac{(1-C)(\epsilon/\delta)}{C} =O(\epsilon)
\]
\DOTHIS{}{these probabilities are a bit confusing}
where $C$ is the constant of Claim \ref{cl:tdv-together-estimate}. As the
behavior of $\twodim$ is independent on those intervals, we deduce Lemma
\ref{lem:Phitboundaryline}.

\DOTHIS{}{Why did we need independence?}
\end{proof}
}
\begin{proof}[Proof of Lemma \ref{lem:Pabsorbedandtravelsfar}]
\newcommand{\rotfarpoint}{(P,0)}
\newcommand{\segment}{[\epsilon,1] \cross \{0\}}
We bound below the probability that an excursion hits $\farpoint$,
i.e.\ $\rotproc$ hits $\rotfarpoint$ before returning to $\origin$.
We do this by considering the probability that the excursion takes
the following form: $\rotproc$ travels from $\origin$ to the line
segment $Q = [0, \epsilon] \cross \{\epsilon\}$, then hits the horizontal
axis for the first time in $\segment$, then travels to $\rotfarpoint$,
before returning to $\origin$.

After a stopping time at which $\twodim = \rotproc = \origin$ there is a positive probability $K$
that $\rotproc$ hits $Q$ before returning to $\origin$.
By scale invariance $K$ is independent of $\epsilon$.

Consider $\rotproc$ after hitting some point in $Q$.  We now
bound the hitting density of this process on the horizontal
axis.  Regardless of the point in $Q$, this density for points on
$\segment$ is at least
\[
\frac{1}{\pi\epsilon} \frac{1}{1 + (x/\epsilon)^2} dx.
\]
This follows directly from the classical result that the hitting density
on a line of the two-dimensional Brownian motion is a Cauchy distribution
(see for example Theorem 2.37 of \cite{mortens-peres}).

On hitting a point $(x,0)$ for $x \in [\epsilon, 1]$ the process
transitions from state $\statewebapart$ to state $\statewebtogether$.
When in state $\statewebtogether$, $\rotproc$ behaves as a
one-dimensional Brownian motion on the horizontal axis until it hits
$\origin$.
By the same martingale argument which justifies
Claim \ref{cl:tdv-hitting-back-0-estimate}, the
probability of subsequently hitting $\rotfarpoint$ before $\origin$ is $x/P$.
Integrating this against the hitting density we get that the probability that
$\rotproc$ started from some point in $Q$ hits the horizontal axis in $\segment$
and then travels to $\rotfarpoint$ before returning
to $\origin$ is at least
\[
\frac{1}{\pi P} \int_{\epsilon}^{1} \frac{x/\epsilon}{1 + (x/\epsilon)^2}
\, dx
=
\frac{\epsilon}{2\pi P} \log\left(\frac{1 + (1/\epsilon)^2}{2}\right)
,\text{ which is }
\Omegaeloge.
\]
\TODO{}{This now looks slightly odd}
\end{proof}
}

  {
\section{Recovering the web from strips}
\label{sec:recovering-from-strips}

The next step towards proving Theorem \ref{thm:bw-2d-black-noise} is
to show that the \sigfield{} associated to a horizontal strip is
generated by those \sigfield{}s associated to any two substrips
which partition the larger strip.  To do so, we follow closely the
\DOTHIS{structure}{what's the best word?} of Section \ref{sec:recovering-from-half-planes}.

\newcommand{\strip}{(-\infty, \infty) \cross (y,z)}

In the same way we defined $\upperhp$ and $\lowerhp$,
we introduce a $\sigma$-field $\factor{y}{z}$ for each $y < z \in
[-\infty, \infty]$, generated by the web in the horizontal strip
$\strip$.  Formally,

\begin{definition}
  \newcommand{\T}{\stripleavetime{f}}
  \label{def:restrict}
  For any path $f$ we write $\restrict{y}{z}{f}$ for $f$ stopped at
  the first time it is outside $\strip$,
  i.e.\ $\restrict{y}{z}{f}(t) = f(t \minsym \T)$ where $\T = \inf\{ t
  : f(t) \not\in (y,z) \}$ (as in Section
  \ref{sec:recovering-from-half-planes}, if $f$ starts
  outside $\strip$, it is stopped immediately).  We define
  $\factor{y}{z}$ to be the $\sigma$-field generated by the collection
  of paths $\{\restrict{y}{z}{X} : X \in \trajs \}$.
  \label{def:horizontal-factorization}
  The association of strips $\strip$ to \sigfield{}s $\factor{y}{z}$ we call
  the \emph{horizontal factorization of the Brownian web}.
\end{definition}

With these definitions, $\wholefield$ of Section
\ref{sec:recovering-from-half-planes} is
$\commafactor{-\infty}{\infty}$, $\upperhp$ is
$\commafactor{0}{\infty}$ and $\lowerhp$ is
$\commafactor{-\infty}{0}$.
Note that
  $\factor{w}{x}$ and $\factor{y}{z}$ are independent if the
  intervals $(w,x)$ and $(y,z)$ are disjoint.

\renewcommand{\top}{b}
\newcommand{\bottom}{a}

\begin{theorem}\label{thm:recoveringfromstrips}
  $\factor{x}{z} = \factor{x}{y} \tensor \factor{y}{z}$ for all $x < y < z$.
\end{theorem}

The Brownian web is translation invariant.  Thus without loss of
generality we can limit ourselves to $x=a$, $y=0$, $z=b$ in the above
theorem, for arbitrary fixed $a$ and $b$.  This reduces the theorem to
\[
\factor{a}{b} = \factor{a}{0} \tensor \factor{0}{b}.
\]
In the rest of this
section we therefore write $\Res{\cdot}$ for
$\restrict{\bottom}{\top}{\cdot}$.

Given the definition of $\factor{a}{b}$ and since $\sampled$ is
arbitrary, Theorem \ref{thm:recoveringfromstrips} is a consequence of
the following.
\begin{lemma}
  \label{lem:res-sampled-strips-meas}
  $\Res{\sampled}$ is $\factor{\bottom}{0} \tensor
  \factor{0}{\top}$-measurable.
\end{lemma}

\newcommand{\Resresamplede}{\Res{\resamplede}}
\newcommand{\Ressampled}{\Res{\sampled}}

Similarly to Section \ref{sec:recovering-from-half-planes}, $\Resresamplede$
is constructed from trajectories that are $\factor{\bottom}{0}$,
$\factor{0}{\top}$ and $\reservoir$-measurable only.
Formally,
  $\Res{\resamplede}$ is $\factor{\bottom}{0} \tensor \factor{0}{\top}
  \tensor \reservoir$-measurable.
Thus as in the reduction of Lemma \ref{lem:sampled-twostrip-meas} to
Lemma \ref{lem:resamplede-to-sampled}, Lemma
\ref{lem:res-sampled-strips-meas} follows from

\begin{lemma}
    \label{lem:resamplede-to-sampled-strip}
    $\Res{\resamplede} \toinP \Res{\sampled}$ as $\epsilon \to 0$.
\end{lemma}

We could show this convergence result
directly by an extension of the argument we used for the half-planes
in Section \ref{sec:recovering-from-half-planes}.  However, knowing that
$\resamplede \toinP \sampled$ is nearly enough, and all that is required
in addition is that this convergence is preserved by $\Res{\cdot}$.
For this we use the following
straightforward result in classical analysis.
\DOTHIS{}{Naomi wants a reference}

\newcommand{\stripleavetimenotturningpoint}[1]{$\stripleavetime{#1}$ is
  not a turning point of the path $#1$}

\begin{lemma}
  \label{lem:Res-continuous-ae}
  If \stripleavetimenotturningpoint{f},
  then the map $f \mapsto \Res{f}$ is continuous at $f$ in the
  topology of uniform convergence on compacts.
\end{lemma}

\newcommand{\commenttom}[1]{}
\commenttom{
\DOTHIS{}{We decided to omit this proof -- ohad and tom
\DOTHIS{We sketch a proof.
\begin{proof}
    \newcommand{\fn}{f_n}

    \newcommand{\T}{\stripleavetime{f}}
    \newcommand{\Tn}{\stripleavetime{\fn}}
    \newcommand{\starttime}{s}

    Suppose $\fn \to f$ uniformly on compacts, and both are paths
    starting at time $\starttime$.

    If $\T = \infty$ there is very little to prove, so suppose
    otherwise.  Then since $\T$ is not
    a turning point of $f$ it is easily seen that $\Tn \to \T$ and
    that eventually \DOTHIS{$\fn(\Tn) = f(\T)$}{They both take the value
      either $\bottom$ or $\top$ -- or both of them start outside the
      strip}.

    \renewcommand{\d}{\delta}
    \newcommand{\e}{\epsilon}

    Since $f$ is continuous, choose $\d$ such that
    $|f(\T) - f(x)| \le \epsilon$ when
    $|\T - x| \le \delta$.  Eventually
    \begin{itemize}
    \item $\fn(\Tn) = f(\T)$
    \item $|\Tn - \T| \le \d$
    \item $|\fn - f| \le \e$ uniformly on some compact containing
      $[\starttime, \T+\d]$.
    \end{itemize}

    \newcommand{\stoppedfn}{\Res{\fn}}
    \newcommand{\stoppedf}{\Res{f}}

    \newcommand{\condition}[2]{for $t \in {#1}$ we have
      $|\stoppedfn(t) - \stoppedf(t)| #2$}
    so
    \begin{itemize}
    \item \condition{[\starttime, \T-\d)}{\le \e}
    \item \condition{[\T-\d, \T+\d]}{\le 2\e}
    \item \condition{(\T+\d, \infty)}{= 0}
    \end{itemize}
    so indeed $\stoppedfn \to \stoppedf$ uniformly on compacts.}
      {The only tricky part is the second item.  There are
      four cases to consider.  The worst case is when $\stoppedf$ has
      been truncated and $\stoppedfn$ hasn't.  But $\stoppedf$ is
      equal to the level, and so is within $\e$ of $f$.  OTOH
      $\stoppedfn = \fn$ so is within $\e$ of $f$.  Thus the $2\e$
      bound.}
\end{proof}
}}

\begin{proof}[Proof of Lemma \ref{lem:resamplede-to-sampled-strip}]
  We know from Lemma \ref{lem:resamplede-to-sampled}
  that $\resamplede \toinP \sampled$.  In addition
  $\sampled$ is a Brownian motion so almost surely
  satisfies the condition of Lemma \ref{lem:Res-continuous-ae}, i.e.\
  ``\stripleavetimenotturningpoint{\sampled}''.
  We conclude that $\Res{\resamplede} \toinP
  \Res{\sampled}$ by the continuous mapping theorem (see for
  example \cite{billingsley} Theorem 2.7).
\end{proof}
}

  {
\section{Conclusions about the noise}
\label{sec:conclusions-about-the-noise}

We conclude by supplying a formal framework for the statement of
Theorem \ref{thm:bw-2d-black-noise} followed by its proof.
The following definition of noise is a straightforward extension of
that due to Tsirelson (Definition
3d1 of \cite{tsirelson-nonclassical-stochastic-flows})
to multiple dimensions.

\newcommand{\F}{\mathcal{F}}
A \emph{$d$-dimensional noise} consists of a probability space $(\Omega,\F,
\P)$, sub-$\sigma$-fields $\F_R \subset \F$ given for all open
$d$-dimensional rectangles $R \subset \R^d$, and a measurable action
$(T_h)_h$ of the additive group of $\R^d$ on $\Omega$, having the following properties:

\begin{enumerate}[(a)]
\item \label{item:tensor-condition} $\F_R \tensor \F_{R'} = \F_{R''}$ whenever
$R$ and $R'$ partition $R''$, in the sense that
$R\cap R'=\emptyset$ and the closure of $R \cup R'$
is the closure of $R''$,
\item \label{item:translation-condition} $T_h$ sends $\F_R$ to $\F_{R+h}$ for each $h \in \R^d$,
\item \label{item:generation-condition} $\F$ is generated by the union of all $\F_R$.
\end{enumerate}

When $d = 1$ our definition coincides with that of Tsirelson.
In that case, $R$ ranges over all open intervals
and condition~(\ref{item:tensor-condition}) translates to
$\F_{(s,t)} \tensor \F_{(t,u)} = \F_{(s,u)}$ whenever $s < t < u$.

As conditions (\ref{item:translation-condition}) and
(\ref{item:generation-condition}) are immediate for
the horizontal factorization of the Brownian web,
Theorem~\ref{thm:recoveringfromstrips} immediately
implies the following:

\begin{proposition*}
The horizontal factorization of the Brownian web is a (one-dimensional)
noise.
\end{proposition*}

Recall that the horizontal factorization of the Brownian web is an
association of a \sigfield{} to any horizontal strip (see
Definition~\ref{def:horizontal-factorization}).
Observe that the association arises from considering trajectories of the Brownian
web stopped at the first time they are outside a particular strip.
Similarly, we can associate a \sigfield{} to any vertical strip, or
indeed to any rectangle.  The former association is the \emph{vertical
  factorization of the Brownian web} and the latter the
\emph{two-dimensional factorization}.

We can extend existing results to derive the following:

\begin{proposition*}
The Brownian web factorized on two-dimensional
rectangles is a two-dimensional noise.
\end{proposition*}

That is, when a rectangle is partitioned horizontally or vertically
into two smaller rectangles, the \sigfield{} of the larger is
generated by the \sigfield{}s of the two smaller.  To see this holds
for a rectangle partitioned by a horizontal split observe that this is
a consequence of our result restricted to a finite time interval.  The
see it holds  for a vertical split, observe that this is a
straightforward modification of the earlier result that the vertical
factorization of the Brownian web is a noise (see, for example,
\cite{tsirelson-scaling-limit-noise-stability}).

Furthermore, by a general result
of Tsirelson (see
\cite{tsirelson-noise-as-a-boolean-algebra} Theorem 1e2),
a two-dimensional noise is black when
one of its one-dimensional factorizations is
black. As the vertical factorization of
the Brownian web is black (see
\cite{tsirelson-nonclassical-stochastic-flows}),
we deduce Theorem \ref{thm:bw-2d-black-noise}.
}

  {
\section{Remarks and open problems}

\label{sec:open-problems}

In this paper we present the second known example of a two-dimensional
black noise, after the scaling limit of critical planar percolation,
as proved by Schramm and Smirnov \cite{schramm-smirnov}.  They also remark
that \sigfield{}s can be associated to a larger class of domains than
just rectangles in a way that still allows the \sigfield{} of a larger
domain to be recovered from the \sigfield{}s of two smaller domains
that partition it.
In particular, in Remark 1.8 they claim that this can be done for the
scaling limit of site percolation on the triangular lattice, as long
as border between those domains has Hausdorff dimension less
than 5/4, and cannot be done if the border has Hausdorff dimension
greater then than 5/4. This raises the following question:

\begin{openproblem}
  \label{openproblem:extend}
  To what class of two-dimensional domains can the Brownian web noise be extended?
\end{openproblem}

We expect the answer to be more sophisticated than for percolation,
since the Brownian web is not rotationally invariant.  
This suggests that Hausdorff dimension is not a sufficient measurement
to determine from which subdomains the Brownian web can be
reconstructed.  In some sense it is easier to reconstruct the Brownian
web from vertical strips than it is from horizontal strips.

We may obtain a better understanding of Open Problem
\ref{openproblem:extend} if we can answer

\begin{openproblem}
  Give an explicit example of domains to which the noise
  cannot be extended.
\end{openproblem}

Analogy with general results in the one-dimensional case suggests that
such domains should exist (see
\cite{tsirelson-nonclassical-stochastic-flows} Theorem 11a2 and
Section 11b).

Moreover, having seen that the Brownian web is a two-dimensional black
noise, further examples in two dimensions (and indeed in higher
dimensions) are called for.  Their
discovery would hopefully shed light on the nature of black noises.

\begin{openproblem}
  Find more examples of two-dimensional black noises.  Show an example
  of a black noise in three dimensions or higher.
\end{openproblem}

Readers may wish to note that in
\cite{tsirelson-noise-as-a-boolean-algebra} Tsirelson extends the
concept of a noise to a much more abstract and general setting.  This
allows results on noises to be formulated and proved without having an
explicit underlying geometric base.  However, our methods here which are
concrete and geometric in nature are more conveniently described in terms
of the earlier formulation.
}

  {
\subsection{Acknowledgements}

The authors would like to thank Boris Tsirelson for his guidance in
the field of noises, for introducing the question to us, and for helpful
discussions.  In addition we would like to thank Ron Peled and Naomi
Feldheim for simplifying some of our arguments and their useful
comments on preliminary drafts.
}

}


\begin{thebibliography}{99}
\bibitem{arratia} Arratia, R.A. (1979). Coalescing Brownian motions on
  the line. PhD thesis, University of Wisconsin, Madison.
\bibitem{billingsley} Billingsley, P. (1999). Convergence of probability
  measures. Wiley Series in Probability and Statistics.
\bibitem{fontes-et-al} Fontes, L. R. G., Isopi, M., Newman, C. M. and
  Ravishankar, K. (2004). The Brownian web: characterization and
  convergence. Ann. Probab. 32(4):2857--2883.
\bibitem{mortens-peres} M\"orters, P. and Peres, Y. (2010). Brownian
  Motion.  Cambridge University Press.
\bibitem{norris-turner} Norris J. and Turner, A. (2008). Planar
  aggregation and the coalescing Brownian
  flow. arXiv:0810.0211v1.
\bibitem{schramm-smirnov}
  Schramm, O. and Smirnov, S. (2011).
  On the scaling limits of planar percolation.
  Ann. Probab. 39(5):1768--1814.
\bibitem{toth-werner} Toth, B. and Werner, W. (1998). The true
  self-repelling motion. Probab. Theory Related Fields
  111(3):375--452.
\bibitem{tsirelson-vershik}
  Tsirelson, B. and Vershik, A. (1998). Examples of nonlinear continuous
  tensor products of measure spaces and non-Fock factorizations. Reviews
  in Mathematical Physics 10(1):81--145.
\bibitem{tsirelson-lecture-course}
  Tsirelson, B. (2002)
  White noises, black noises and other scaling limits.
  Lecture course, Tel Aviv University.
  http://www.tau.ac.il/\textasciitilde{}tsirel/Courses/Noises/syllabus.html
\bibitem{tsirelson-scaling-limit-noise-stability}
  Tsirelson, B. (2004).
  Scaling Limit, Noise, Stability.
  Lecture Notes in Mathematics 1840, Springer: 1--106.
\bibitem{tsirelson-nonclassical-stochastic-flows} Tsirelson, B. (2004).
  Nonclassical stochastic flows and continuous products.
  Probability Surveys 1: 173--298.
\bibitem{tsirelson-noise-as-a-boolean-algebra}
  Tsirelson, B. (2011).
  Noise as a Boolean algebra of sigma-fields.
  arXiv:1111.7270v1.
\end{thebibliography}
\end{document}